\documentclass[12pt,a4paper,psamsfonts]{amsart}
\usepackage[english]{babel}
\usepackage[T1]{fontenc}
\usepackage{fancyhdr}
\usepackage{appendix}
\usepackage{amssymb,amscd,amsxtra,calc}
\usepackage{mathrsfs}
\usepackage{cmmib57}
\usepackage{multirow}
\usepackage[all]{xy}
\usepackage{longtable}
\usepackage[colorlinks=true,linkcolor=blue,anchorcolor=blue,citecolor=blue,pagebackref,linktocpage]{hyperref}
\usepackage{cleveref}
\usepackage{tikz}
\usepackage{cite}
\theoremstyle{plain}
    \newtheorem{thm}{Theorem}[section]

     \newtheorem{conjecture}[thm]{Conjecture}
    \newtheorem{corollary}[thm]{Corollary}

    \newtheorem{proposition}[thm]{Proposition}
    \newtheorem{question}[thm]{Question}
     \newtheorem{problem}[thm]{Problem}
    \newtheorem{theorem}[thm]{Theorem}

\theoremstyle{definition}
    \newtheorem{definition}[thm]{Definition}
     
    \newtheorem{assumption}[thm]{Assumption}
    
    \newtheorem*{notation*}{Notation and Terminology}
      
    \newtheorem{remark}[thm]{Remark}
    
\theoremstyle{remark}

\newcommand{\arxiv}[1]{\href{https://arxiv.org/abs/#1}{{\tt arXiv:#1}}}

\newcommand{\bP}{\mathbb{P}}
\newcommand{\bQ}{\mathbb{Q}}
\newcommand{\bR}{\mathbb{R}}
\newcommand{\bC}{\mathbb{C}}

\newcommand{\NE}{\overline{\operatorname{NE}}}
\newcommand{\Nef}{\operatorname{Nef}}

\newcommand{\NS}{\operatorname{NS}}

\newcommand{\mstriangle}[1]{
\begin{tikzpicture}[x=0.3cm,y=0.3cm]
\draw (-0.4,-0.433) -- (1.4,-0.433);
\draw (-0.2,-0.7794) -- (0.7,0.7794);
\draw (1.2,-0.7794) -- (0.3,0.7794);
\end{tikzpicture}
}
\newcommand{\mssharp}[1]{
\begin{tikzpicture}[x=0.3cm,y=0.3cm]
\draw (-0.8,-0.5) -- (0.8,-0.5);
\draw (-0.8,0.5) -- (0.8,0.5);
\draw (-0.5,-0.8) -- (-0.5,0.8);
\draw (0.5,-0.8) -- (0.5,0.8);
\end{tikzpicture}
}

\makeatletter

\newcommand{\Rmnum}[1]{\expandafter\@slowromancap\romannumeral #1@}
\makeatother

\begin{document}
\title[Bounded cohomology property]
{Bounded cohomology property on Jacobian elliptic surfaces with simplicial Mori cones}
\author{Sichen Li}
\address{
School of Mathematics, East China University of Science and Technology, Shanghai 200237, P. R. China}
\email{\href{mailto:sichenli@ecust.edu.cn}{sichenli@ecust.edu.cn}}

\begin{abstract}
Let $X$ be a Jacobian elliptic surface with finite Mordell-Weil group and exactly one reducible fiber.
We prove that if  $\chi(\mathcal O_X)\ge \rho(X)$, then the Mori cone of $X$ is simplicial.
As an application, assuming additionally $q(X)=0$, we  show that $X$ satisfies the bounded cohomology property (BCP): there exists a constant $c_X>0$ such that
$
h^1(\mathcal O_X(C))\le c_X h^0(\mathcal O_X(C))
$
for every curve $C$ on $X$.
We also establish a necessary and sufficient condition for the BCP to hold on minimal smooth projective surfaces $Y$ with $\kappa(Y)\ge 1$, $q(Y)=0$, and rational polyhedral Mori cones.
\end{abstract}
\maketitle
\section{Introduction}
The Bounded Negativity Conjecture (BNC for short)  is one of the most subtle topological problems in the theory of algebraic surfaces and can be formulated as follows.
\begin{conjecture}
\cite[Conjecture 1.1]{Bauer et al. 2013} 
Let $X$ be a smooth projective surface. Then there exists an integer $b=b(X)\ge0$ such that $C^2\ge-b$ for every curve $C\subseteq X$.
\end{conjecture}
It is well-known that the following SHGH Conjecture implies Nagata's Conjecture \cite{Nagata59}, which is motivated by Hilbert's 14th problem (cf. \cite[Lemma 2.4]{CHMR13}).
\begin{conjecture}
\cite[Conjecture 2.5.1]{Bauer et al. 2012} 
Let $C\subseteq X$ be a curve where $X\to\bP^2$ is the blow-up of general points $p_1,\cdots, p_n$ with $n\ge10$.
Then $h^1(\mathcal O_X(C))=0$.
\end{conjecture}
Motivated by the BNC and the SHGH Conjecture,  Bauer et al. \cite[Conjecture 2.5.3]{Bauer et al. 2012} introduced the following terminology.
\begin{definition}
A smooth projective surface $X$ is said to satisfy the bounded cohomology property (BCP for short), if there exists a constant $c_X>0$ such that $h^1(\mathcal O_X(C))\le c_Xh^0 (\mathcal O_X(C))$ for every curve $C$ on $X$.
\end{definition}
It turns out by Ciliberto et al. \cite{Ciliberto et al. 2017} that the BCP implies the BNC.
Note that a smooth projective surface satisfies the BCP provided that  $-K_X$ is psuedoeffective (cf \cite[Proposition 1.4]{HL26}).
In general, it is  an open problem to classify smooth projective surfaces with the BCP as posted by Ciliberto et al.\cite[Question 6]{Ciliberto et al. 2017}.
\begin{problem}
Classify all smooth projective surfaces with the BCP.
\end{problem}
 In \cite[Claim 2.11]{Li19}, Li listed a classification result of smooth projective surfaces $X$ with Picard number $\rho(X)=2$ and two negative curves, where the closed Mori cone $\NE(X)$ is rational polyhedral.
Later \cite{HL26, Li21, Li23, Li26} characterized  the BCP on smooth projective surfaces  $X$ with $\rho(X)=2$ where either $\NE(X)$ is rational polyhedral or $X$ is a geometrically ruled surface.
The classification of the BCP for surfaces with higher Picard numbers remains a widely open problem.
In this setting, one is generally forced to assume that $\NE(X)$ is rational polyhedral as follows.
\begin{question}
\cite[Question 1.7]{HL26}
\label{Que-Li}
Does every smooth projective surface $X$ satisfy the BCP provided that $\NE(X)$ is rational polyhedral?
\end{question}
Recently, Hua and Li \cite{HL26} established the BCP for Mori dream surfaces, for which $q(X)=0$, $\NE(X)$ is rational polyhedral and every nef divisor is semiample.
Their proof hinges on the condition that every nef extremal ray has positive Iitaka dimension.
Furthermore, in this paper we give the following necessary and sufficient condition for the BCP to hold on minimal smooth  projective surfaces $X$ satisfying $q(X)=0, \kappa(X)\ge1$ and whose $\NE(X)$ is rational polyhedral.
\begin{theorem}
\label{nef-positive-thm}
Let $X$ be a smooth projective surface with $q(X)=0$ such that $\NE(X)$ is rational polyhedral.
Let $\Nef(X)=\sum_{i=1}^n\bR_{\ge0}[C_i]$  where each $C_i$ is an effective divisor on $X$.
Then the following statements holds.
\begin{enumerate}
\item[(a)] Suppose each $C_i$ has $\kappa(X,C_i)\ge1$.
Then $X$ satisfies the BCP.
\item[(b)] 
Suppose $\kappa(X)\ge 1$ and $K_X$ is nef. Then $X$ satisfies the BCP if and only if for any indices $i\neq j$ and any curve $D=a_iC_i+a_jC_j$ with $$a_i>0, \quad 0<2a_j(C_i\cdot C_j)<(K_X\cdot C_i),$$we have
$$
\liminf_{a_i\to+\infty}\frac{h^0(\mathcal O_X(a_iC_i+a_jC_j))}{a_i}>0.
$$
\end{enumerate}
\end{theorem}
\begin{remark}
Let $X$ be a smooth projective surface with $q(X)=0$ and rational polyhedral Mori cone.
Theorem \ref{nef-positive-thm} relates the BCP for $X$ to its nef extremal rays. 
Part (b) reformulates the BCP as an asymptotic condition on section-growth rates for pairs built from these rays, demonstrating that numerical constraints alone are generally insufficient.
By contrast, it is straightforward to establish the BCP for $X$ whenever every nef extremal ray is spanned by a divisor $D_i$ with $\kappa(X,D_i)\ge 1$, as in part (a).
\end{remark}
Since a rational polyhedral Mori cone implies the BNC (cf. \cite[Proposition 1.1]{AL11}), characterizing surfaces with rational polyhedral Mori cones is a fundamental problem.
\begin{problem}
\label{MainProb}
Classify smooth projective surfaces with rational polyhedral Mori cone.
\end{problem}
\begin{remark}
For smooth projective surfaces with rational polyhedral Mori cone, Nikulin \cite{Nikulin00} proved a crucial boundedness theorem: fixing the Picard number \(\rho\), maximal negative self-intersection $$\delta_E(X):=\max\big\{ -C^2 ~\big |~ C \text{ is a  curve on } X\big\},$$ and maximal arithmetic genus \(p_E(X)\) of exceptional curves restricts such surfaces to finitely many isomorphism classes.
 He gave an explicit full classification only for the subcase that \(\delta_E(X)\le 2\) and \(p_E(X)=0\).
  For surfaces with \(\delta_E(X)\ge3\), the complete classification stated in Problem \ref{MainProb} stays largely open.	
\end{remark}
Let $\pi: X\to B$ be a Jacobian elliptic surface.
It is straightforward to show the Mordell-Weil group $\mathrm{MW}(\pi)$ is finite provided that $\NE(X)$ is rational polyhedral (cf. \cite[Proposition 5.14]{HL26}).
Conversely,  finiteness of  $\mathrm{MW}(\pi)$ does not imply rational polyhedrality of $\NE(X)$  in general, see \cite[Theorem 3.2]{LLL26}.
Nevertheless, suppose $\mathrm{MW}(\pi)$ is finite, $\pi$ has exactly one reducible fiber of type $I_n$, and $4\chi(\mathcal O_X)\ge n$.
Then $\NE(X)$ is simplicial by \cite[Corollary 1.4(1)]{LLL26}.
Examples of Jacobian elliptic surfaces satisfying these hypotheses, are given in \cite[Theorems 1.6]{LLL26}.
Building upon these observations, we prove the following result.
\begin{theorem}
\label{Jac-thm}
Let $X$ be a Jacobian elliptic surface with finite Mordell-Weil group and  exactly one reducible fiber. 
\begin{enumerate}
	\item If $\chi(\mathcal O_X)\ge\rho(X)$, then $\NE(X)$ is simplicial. 	
	\item If  $\chi(\mathcal O_X)\ge\rho(X)$ and $q(X)=0$, then $X$ satisfies the BCP.
\end{enumerate}
\end{theorem}
Let $\pi\colon X\to \mathbb P^1$ be a semistable Jacobian elliptic surface, with singular fibres $F_1,\dots,F_s$ of types $I_{n_1},\dots,I_{n_s}$, and let $\mathcal T$ be the set of their irreducible components.
If $C_0$ denots the zero section, Laface et al. \cite{LLL26} considered the fibration cone
$$
\mathcal C_\pi=\bR_{\ge0}[C_0]+\sum_{T\in \mathcal T} \bR_{\ge0}[T].
$$
Let $F$ denote a general fibre of $\pi$.
Set
\[
\delta(\pi):=\sum_{i=1}^s \frac{\lfloor n_i^2/4\rfloor}{n_i}.
\]
Laface et al.~\cite[Theorem~1.1(1)]{LLL26} proved that $\mathcal C_\pi=\NE(X)$ if and only if $\delta(\pi)\le \chi(\mathcal O_X)$.
Note that if $\delta(\pi)<\chi(\mathcal O_X)$, then every nef extremal ray is either spanned by $F$ or by a nef and big divisor by \cite[Proof of Theorem 1.1(1)]{LLL26}.
Thus the following result is an immediate consequence of Theorem~\ref{nef-positive-thm} (a).
\begin{corollary}
\label{semistable-coro}
Let $\pi: X \to \bP^1$ be a semistable Jacobian elliptic surface.
If $\delta(\pi)<\chi(\mathcal O_X)$, then $X$ satisfies the BCP.
\end{corollary}
\begin{remark}
When $\delta(\pi)=\chi(\mathcal O_X)$, it is unknown whether  $\kappa(X,D)\ge1$ holds for every divisor $D$ spanning a nonvertical extremal ray of $\Nef(X)$, see~\cite[Problem~1.3]{LLL26}.
\end{remark}
\subsection*{Acknowledgements}
The author would like to thank Antonio Laface for answering questions.
\section{Preliminaries}
\label{Pre}
{\bf Notation and Terminology.}
In this paper, $X$ is a smooth projective surface over $\bC$.
\begin{itemize}
\item By a curve on $X$, we mean a reduced and irreducible curve.
\item A negative curve on $X$ is a curve with negative self-intersection.
 \item A prime divisor $C$ on $X$ is either a nef curve or a negative curve (in the latter case, $h^0(\mathcal O_X(C))=1$).
 \item For every $\bR$-divisor $C$ with $C^2\ne0$ on $X$, we define a value $l_C$ associated to $C$ as follows:
\begin{equation*}
                                                       l_C:=\frac{(K_X\cdot C)}{\max\bigg\{ 1, C^2\bigg\}}.
\end{equation*}
\item For every $\bR$-divisor $C$ with $C^2=0$ on $X$, we define a value $l_C$ associated to $C$ as follows:
\begin{equation*}
                                   l_C:=\frac{(K_X\cdot C)}{\max\bigg\{1,h^0(\mathcal O_X(C))\bigg\}}.
\end{equation*}
\end{itemize}
\begin{definition}
\cite[Definition 2.1]{HL26}
 \label{Defn-l_C}
We say a smooth projective surface $X$ admits uniform boundedness of $l_C$ if there exists a positive constant $m(X)$ such that $l_C\le m(X)$ for every curve $C$ on $X$.
\end{definition}
\begin{definition}
We say the closed Mori cone \(\NE(X)\) is rational polyhedral if there exist finitely many curves $C_1,\cdots, C_s$ on $X$ such that $\NE(X)=\sum \bR_{\ge0}[C_i]$.
And we say $\NE(X)$ is simplicial if it is rational polyhedral and $I(C_1,\cdots, C_n)$ is a non-singular matrix.
\end{definition}
\begin{proposition}
\label{Num-prop}
\cite[Proposition 1.5]{HL26}
Let $X$ be a smooth projective surface.
If $X$ admits the uniform boundedness of $l_C$, and there exists a positive constant $m(X)$ such that either $|C^2|\le m(X)h^0(\mathcal O_X(C))$ or $h^1(\mathcal O_X(C))\le m(X)h^0(\mathcal O_X(C))$ for every curve $C$ on $X$, then $X$ satisfies the BCP.
\end{proposition}

The following result is due to Serre duality.
\begin{proposition}\label{Serre}
Let $C$ be a curve on a smooth projective surface $X$.
Then
\begin{equation*} 
 h^2(\mathcal O_X(C))-\chi(\mathcal O_X)\le q(X)-1.	
\end{equation*}
\end{proposition}
\section{Proof of Theorems \ref{nef-positive-thm} and \ref{Jac-thm}}
\subsection{Proof of Theorem \ref{nef-positive-thm}}
In this subsection, we prove part (a) of Theorem \ref{nef-positive-thm}, following the argument given in \cite[Proof of Theorem 1.9]{HL26}.
\begin{proposition}
\label{q=0-uniform-prop}
Let $X$ be a smooth projective surface with $q(X)=0$.
Then $X$ admits the uniform boundedness of $l_C$ provided that $\NE(X)$ is rational polyhedral.
\end{proposition}
\begin{proof}
To show the uniform boundedness of $l_C$, by \cite[Proposition 2.8]{Li23}, we may take a nef divisor   $D$ with $D^2=0$.
Note that $\Nef(X)=\sum_{i=1}^k \bR_{\ge0}[C_i]$ and each $C_i$ is a nef divisor since $\NE(X)$ is rational polyhedral.
Note that $C_i\neq rC_j$ with $i\ne j$ for any $r\in \bQ_{>0}$.
Then $C_i\cdot C_j>0$ by \cite[Proposition 5.1.1.1]{ADHL15}.
As a result, for a nef curve $D=\sum_{i=1}^r a_iC_i$ with each $a_i>0$, we have $D^2>0$ if $r\ge2$.
So $D=a_kC_k$ with $a_k>0$ since $D^2=0$.
Then there exists a positive constant $m(X)$ such that $l_D\le m(X)$ for every nef curve $D$ with $D^2=0$ by \cite[Proposition 2.7]{Li23}.
In all, $X$ admits the uniform boundedness of $l_C$.
\end{proof}
\begin{proof}[Proof of part (a) of Theorem \ref{nef-positive-thm}]
To show the BCP for $X$,  by Propositions \ref{Num-prop} and \ref{q=0-uniform-prop}, it suffices to show there exist a constant $c_X>0$ such that either $D^2\le c_Xh^0(\mathcal O_X(D))$ or $h^1(\mathcal O_X(D))\le c_Xh^0(\mathcal O_X(D))$ for every curve $D$ with $D^2>0$  on $X$.
Note that $\Nef(X)=\sum_{i=1}^k \bR_{\ge0}[C_i]$ where  each $C_i$ is an effective divisor with $\kappa(X,C_i)\ge1$.
Take $D=\sum_{i=1}^r a_iC_i$ with each $a_i>0$.
If $k=1$, then $\rho(X)=1$ and every nef divisor is ample.
Then $X$  satisfies the BCP by \cite[Proposition 2.6]{Li23}.

Now we assume that $k\ge2$.
Take a nef and big curve $D=\sum_{i=1}^r a_iC_i$ with each $a_i>0$.
Without loss of generality, we may assume that \(a_1\ge a_2\ge\cdots \ge a_r\) and \(a_1\gg 1\).
If \(a_1\) were bounded above by $k_X$, then 
\begin{equation*}
\begin{split}
D^2 &=\sum_{i=1}^r a_i^2C_i^2+\sum_{1\le i<j\le r}2(C_i\cdot C_j)
 \\ & \le  c_Xh^0(\mathcal O_X(D)),
\end{split}
\end{equation*}
where $$c_X:=k_X^2(r+r^2)\max\big\{C_i^2,\; C_i\cdot C_j \;\big|\; 1\le i\le r,\; 1\le i<j\le r\big\}.$$
 So we may assume that  \(a_1\gg 1\) to complete the proof.
 
Note that
\begin{equation}
\begin{split}
\label{MDS-eq1}
	(K_X-D)D&\le \sum_{i=1}^r a_i(K_X\cdot C_i)
	\\&\le \bigg(\sum_{i=1}^r|K_X\cdot C_i|\bigg)a_1.
\end{split}
\end{equation}
 By \cite[Corollary 2.1.38]{Lazarsfeld04} and $\kappa(X,C_1)\ge1$, for $a_1\gg1$, there exist two positive constants $c=(2c_X-2)^{-1}\sum_{i=1}^r|K_X\cdot C_i|$ and $c_X>1$  (which are independent of the choice of $D$) such that 
\begin{equation}
\label{MDS-eq2}
	 h^0(\mathcal O_X(D))\ge h^0(\mathcal O_X(a_1C_1))\ge\frac{1}{2c_X-2}\bigg(\sum_{i=1}^r|K_X\cdot C_i|\bigg)a_1.
\end{equation}
Then by Riemann-Roch Theorem, (\ref{MDS-eq1}), (\ref{MDS-eq2}) and Proposition \ref{Serre}, we have 
\begin{equation*}\begin{split}
	  h^1(\mathcal O_X(D))&=h^0(\mathcal O_X(D))+h^2(\mathcal O_X(D))+\frac{(K_X-D)\cdot D}{2}-\chi(\mathcal O_X)\\&\le(c_X+q(X))h^0(\mathcal O_X(D))
	  \\&\le c_Xh^0(\mathcal O_X(D)),
\end{split}
 \end{equation*}
 where we use $q(X)=0$.
 In all, $X$ satisfies the BCP.
\end{proof}
\begin{proof}[Proof of part (b) of Theorem \ref{nef-positive-thm}]
 ($\Leftarrow$)
To show the BCP for $X$,  by Propositions \ref{Num-prop} and \ref{q=0-uniform-prop}, it suffices to show there exist a constant $c_X>0$ such that either $D^2\le c_Xh^0(\mathcal O_X(D))$ or $h^1(\mathcal O_X(D))\le c_Xh^0(\mathcal O_X(D))$ for every curve $D$ with $D^2>0$  on $X$.
Note that $\Nef(X)=\sum_{i=1}^k \bR_{\ge0}[C_i]$ with each curve $C_i$ with $\kappa(X,C_i)\ge1$.
Take $D=\sum_{i=1}^r a_iC_i$ with each $a_i>0$.
If $k=1$, then $\rho(X)=1$, and every nef divisor is ample.
Then $X$  satisfies the BCP by \cite[Proposition 2.6]{Li23}.

Now we assume that $k\ge2$.
Take a nef and big curve $D=\sum_{i=1}^r a_iC_i$ with each $a_i>0$.
Without loss of generality, we may assume that \(a_1\ge a_2\ge\cdots \ge a_r\) and \(a_1\gg 1\).
If \(a_1\) were bounded above by $k_X$, then 
\begin{equation*}
\begin{split}
D^2 &=\sum_{i=1}^r a_i^2C_i^2+\sum_{1\le i<j\le r}2(C_i\cdot C_j)
 \\ & \le  c_Xh^0(\mathcal O_X(D)),
\end{split}
\end{equation*}
where $$c_X:=k_X^2(r+r^2)\max\big\{C_i^2,\; C_i\cdot C_j \;\big|\; 1\le i\le r,\; 1\le i<j\le r\big\}.$$
 So we may assume that  \(a_1\gg 1\).
 
  By  Riemann-Roch theorem, we have
\begin{equation}
\begin{split}
\label{BCP-eq0}
	h^1(\mathcal O_X(D))&=h^0(\mathcal O_X(D))+\frac{(K_X-D)\cdot D}{2}+h^2(\mathcal O_X(D))-\chi(\mathcal O_X)
\end{split} 
\end{equation}
 Since $X$ is minimal and $\kappa(X)\ge1$, then 
 \begin{equation}
 \label{K_X-eq}
 K_X\cdot C_i>0, ~i=1,\cdots, k; \quad C_i\cdot C_j>0, ~1\le i<j\le k;	
 \end{equation}
by \cite[Proposition 5.1.1.1]{ADHL15}.
Note that
\begin{equation}
\begin{split}
\label{BCP-eq1}
	(K_X-D)\cdot D&=\sum a_i(K_X\cdot C_i)-\sum a_i^2C_i^2-\sum_{ i<j}2a_ia_j(C_i\cdot C_j).
	\end{split}
\end{equation}
If $a_2\gg1$, then by (\ref{K_X-eq}) and (\ref{BCP-eq1}), $(K_X-D)\cdot D\le0$.
 By (\ref{BCP-eq0}) and Proposition \ref{Serre}, we have $h^1(\mathcal O_X(D))\le h^0(\mathcal O_X(D))$.
Now we assume that $a_2\le N_1$.
If $2a_2(C_1\cdot C_2)\ge (K_X\cdot C_1)$, then (\ref{BCP-eq1}) implies that
\begin{equation}
\begin{split}
\label{KD-I}
	(K_X-D)D &\le a_1((K_X\cdot C_1)-2a_2(C_1\cdot C_2))+\sum_{j=2}^na_j(K_X\cdot C_j)
	\\&\le N_1\sum_{j=2}^n(K_X\cdot C_j).
\end{split}
\end{equation}
 (\ref{BCP-eq0}), (\ref{KD-I}) and Propostion \ref{Serre} imply that 
$$
2h^1(\mathcal O_X(D))\le (2+N_1\sum_{j=2}^n(K_X\cdot C_j))h^0(\mathcal O_X(D)).
$$
Now we may assume that $2a_2(C_1\cdot C_2)<(K_X\cdot C_1)$.
We now use the inequality
\begin{equation}
\label{BCP-eq2}
\liminf_{a_1\to+\infty}\frac{h^0(\mathcal O_X(a_1C_1+a_2C_2))}{a_1}>0 \textup{ with }  0<2a_2(C_1\cdot C_2)<(K_X\cdot C_1)
\end{equation}
to finish the proof of this part.
By (\ref{BCP-eq2}), there is a positive integer $N_2>0$ and constant $c_X>0$, such that for each $a_1\ge N_2$, we have
\begin{equation}
\label{BCP-eq3}
	 h^0(\mathcal O_X(a_1C_1+a_2C_2))\ge (c_X-1)^{-1}ta_1.
\end{equation}
By (\ref{K_X-eq}) and (\ref{BCP-eq1}), $(K_X-D)\cdot D\le ta_1$.
Let $t:=\sum K_X\cdot C_i$.
Then by  (\ref{BCP-eq0}), (\ref{BCP-eq3}) and Proposition \ref{Serre}, we have 
\begin{equation*}\begin{split}
	  h^1(\mathcal O_X(D))&=h^0(\mathcal O_X(D))+h^2(\mathcal O_X(D))+\frac{(K_X-D)\cdot D}{2}-\chi(\mathcal O_X)\\&\le
	  h^0(\mathcal O_X(D))+ta_1
	  \\&\le c_Xh^0(\mathcal O_X(D)).
\end{split}
 \end{equation*}

 $(\Rightarrow)$. 
 Since $X$ satisfies the BCP, there exists a constant $c_X>1$ such that for every curve $D=a_iC_i+a_jC_j$ with $a_i>0$ and $0<2a_j(C_i\cdot C_j)< (K_X\cdot C_i)$,  we have
 \begin{equation}
 \label{BCP-eq4}
  h^1(O_X(D))\le c_Xh^0(\mathcal O_X(D)).
 \end{equation}
Note that 
\begin{equation}
\label{KD}
(K_X-D)\cdot D=a_i((K_X\cdot C_i)-2a_j(C_i\cdot C_j))+a_j(K_X\cdot C_j)-a_j^2C_j^2
\end{equation}
Since $2a_j(C_i\cdot C_j)<(K_X\cdot C_i)$, (\ref{KD}) implies that
 \begin{equation}
 \label{BCP-eq5}
  \liminf_{a_i\to+\infty}\frac{(K_X-D)D}{a_i}>0.
 \end{equation}
So (\ref{BCP-eq0}), (\ref{BCP-eq4}), (\ref{BCP-eq5})  and Proposition \ref{Serre} imply that 
\begin{equation*}
\liminf_{a_i\to\infty}\frac{h^0(\mathcal O_X(a_iC_i+a_2C_i)}{a_i}\ge \frac{1}{(c_X-1)}\liminf_{a_i\to+\infty}\frac{(K_X-D)\cdot D}{a_i}>0.
\end{equation*}
Therefore, we completes the proof of Theorem \ref{nef-positive-thm} (b).
\end{proof}
\subsection{Proof of Theorem \ref{Jac-thm}}
\begin{definition}
 Let $X$ be a smooth projective surface.
Let $\pi: X \to B$ be a surjective morphism with connected fibres.
We call $\pi: X\to B$ a fibration on the surface $X$ with the base curve $B$.
We say $X$ is an \emph{elliptic surface} if the general fiber of $\pi$ is an elliptic curve, and $\pi$ is \emph{relatively minimal} if no fiber contains a $(-1)$-rational curve.
 \end{definition}
For convenience,  all elliptic surfaces throughout this paper are assumed to be relatively minimal, and we will always exclude the trivial product case (cf. \cite[Example 5.6]{SS19})  by the following assumption:
\begin{assumption}
Any elliptic surface $\pi: X\to B$ is relatively minimal, and $\pi$ has a singular fibre.
Thus, we have that $\chi(\mathcal O_X)>0$ by \cite[Corollary 5.50]{SS19}.
\end{assumption}
 \begin{definition}
 Given an algebraic surface fibration $\pi: X\to B$, a \emph{section of $\pi$} is a morphism $\sigma: B \to X$ such that $\pi\circ\sigma$ is the identity map of $B$.
 An elliptic surface $\pi: X\to B$  is \emph{Jacobian} if $\pi$ has a section. 
 Let $\mathrm{MW}(\pi)$ be the Mordell-Weil group of $\pi: X\to B$, i.e., the group of sections, $O$ being the zero section.
 \end{definition}
 \begin{proposition}
\label{jac-rho-prop}
\cite[Proposition 2.4]{LLL26}
Let $\pi\colon X\to B$ be a Jacobian elliptic fibration.
Let $F_1,\ldots,F_s$ be the reducible fibers of $\pi$, and let
$C_{i,0},\ldots,C_{i,m_i-1}$ be the irreducible components of $F_i$, where
$C_{i,0}$ is met by the zero section $C_0$. Let $F$ be a
general fiber. If $\mathrm{MW}(\pi)$ is finite, then
\[
\rho(X)=2+\sum_{i=1}^s(m_i-1),
\]
and the classes of $C_0,F$, and $C_{i,j}$, with
$1\le i\le s$ and $1\le j\le m_i-1$, form a basis of
$\NS_\bQ(X)$.
\end{proposition}
\begin{proof}[Proof of Theorem \ref{Jac-thm}]
Let $F$ be the general fiber.
Let $\chi=\chi(\mathcal O_X)$ and $\rho=\rho(X)$.
Note that $\rho\ge3$ since $\pi$ has exactly one reducible fiber.
Let $C_0$ be the zero section, let $F_1$ be the uniquely reducible fiber.
Note that $C_0^2=-\chi$ by \cite[Corollary 5.45]{SS19}.
Let $C_1,\cdots, C_n$ be all irreducible components of $F_1$ and $C_1\cdot C_0=1$.
By Proposition \ref{jac-rho-prop}, $\rho(X)=n+1$ and $\NS_\bQ(X)$ is generated by $\bQ$-linear independent curve $C_0,F$ and all $C_2,\cdots, C_n$.
Then $I(C_0, F,C_2,\cdots, C_n)$ splits into a two dimensional block generated by $C_0$ and $F$, and  $-Q$, where $Q=(q_{ij})$ is the Cartan matrix with  with $C_2,\cdots, C_n$, where $q_{ik}=-(C_i\cdot C_k)$.
According to Kodaira's classification of singular fibers (cf.\ \cite{Kodaira63}, see also \cite[Definition 7.20]{Esole17}), $Q$ is of type $A_{n-1}$, $D_{n-1}$, $E_6$, $E_7$, or $E_8$.

Now write \(F=C_1+\sum_{i\ge 2}m_i C_i\), where the multiplicities \(m_i\) depend on the type of the reducible fibre.
The divisor $H=C_0+\chi F$ is nef and orthogonal to $C_0,C_2,\cdots, C_s$.
Thus, it is dual to the facet opposite $C_1$.
The fiber is  similarly dual to the facet opposite $C_0$.
For $j\ge2$, the divisor dual to the facet opposite $C_j$ written as $D_j=m_jH-W_j$, where $W_j$ is the unique $\bQ$-divisor supported on $C_2,\cdots, C_n$ such that $W_j\cdot C_k=-\delta_{jk}$ for every $k\ge 2$.
Then  we have
\begin{equation}
\label{W_j-eq}
W_j=\sum_{i\ge2}(Q^{-1})_{ij}C_i.
\end{equation}
Substituting $H = C_0+\chi F$, $F=C_1+\sum_{i\ge2}m_iC_i$ and (\ref{W_j-eq}) into the  the expression $D_j = m_j H - W_j$, we obtain
\[
D_j
= m_j C_0
+\chi m_j C_1
+\sum_{i\ge 2}(\chi m_i m_j-(Q^{-1})_{ij})C_i.
\]
The coefficient of $C_i$ for $i\ge 2$ in $D_j$ is therefore $\chi\, m_i m_j -(Q^{-1})_{ij}$.
Note that 
Nonnegativity of all coefficients of $D_j$ is equivalent to the inequalities
\begin{equation}
\label{coe-iq}
	\chi\, m_i m_j \ge (Q^{-1})_{ij} \quad \text{for all} \quad i,j\ge 2.
\end{equation}
By the root-system correspondence, the intersection matrix $Q$ of components of the reducible fiber is the Cartan matrix of the finite Dynkin type $A_{n-1}$, $D_{n-1}$, $E_6$, $E_7$, or $E_8$.
The fundamental dominant weights, whose expansions in simple roots are tabulated in \cite[\S13.2, Table 1, p.\,69]{Humphreys72} provide the columns of $Q^{-1}$.
We obtain the following bounds for the entries $(Q^{-1})_{ij}$:
\begin{enumerate}
\item[(i)] If $Q$ is of type $A_{n-1}$, then $\rho=n+1$ and $$(Q^{-1})_{ij}=\frac{(\min\{i,j\}-1)(n-\max\{i,j\}+1)}{n}\le \frac{n}{4}<\rho.$$
\item[(ii)] If $Q$ is of type $D_{n-1}$, then $\rho=n+1$ and $$
(Q^{-1})_{ij}=\min\{i,j\}, 1\le j\le n-2; \quad (Q^{-1})_{i,n-1}=\frac{i}{2}.
$$
As a result, $(Q^{-1})_{ij}\le n-2< \rho$.
\item[(iii)] If $Q$ is of type $E_6$, then $\rho=8$ and  $(Q^{-1})_{ij}\le 6$.
\item[(iv)] If $Q$ is of type $E_7$, then $\rho=9$, $m_im_j\ge2$ and  $(Q^{-1})_{ij}\le 12$.
\item[(v)] If $Q$ is of type $E_8$, then $\rho=10$, $m_im_j\ge4$ and $(Q^{-1})_{ij}\le 30$.
\end{enumerate}
Thus, if $\chi\ge\rho$, by (\ref{coe-iq}), the inverse of $A:=I(C_0,C_1,\cdots, C_m)$ is  coefficientwise positive.
Let $\Gamma$ be an irreducible curve different from the $C_i$, and put $v=(\Gamma\cdot C_0,\cdots, \Gamma\cdot C_m)^t$.
All entries of $v$ are nonnegative.
If $[\Gamma]=\sum a_i[C_i]$, then $a=A^{-1}v$, so every $a_i\ge0$.
Thus, $\NE(X)=\sum_{i=0}^m\bR_{\ge0}[C_i]$, and it is simplicial.

Now we assume that $\chi\ge\rho$ and $q(X)=0$.
We  note that 
\begin{equation*}
	\Nef(X)=\bR_{\ge0}[F]+\bR_{\ge0}[H]+\sum_{j\ge2}\bR_{\ge0}[D_j], \quad \kappa(X,F)=1, \quad \kappa(X,H)=\kappa(X,D_j)=2.	
\end{equation*}
Then $X$ satisfies the BCP by Theorem \ref{nef-positive-thm} (a).
\end{proof}


\begin{thebibliography}{99}

\bibitem{ADHL15}
I. Arzhantsev, U. Derenthal, J. Hausen, and A. Laface,
\emph{Cox Rings}, 
Cambridge Studies in Advanced Mathematics, \textbf{144}, Cambridge University Press, 2015.

\bibitem{AL11}
M. Artebani and A. Laface,
\emph{Cox rings of surfaces and the anticanonical Iitaka dimension},
Adv. Math. \textbf{226} (2011), no.~6, 5252--5267.

\bibitem{Bauer et al. 2012}
T. Bauer, C. Bocci, S. Cooper, S. D. Rocci, M. Dumnicki, B. Harbourne, K. Jabbusch, A. L. Knutsen,
A. K$\mathfrak{\ddot{u}}$ronya, R. Miranda, J. Ro$\mathrm{\acute{e}}$, H. Schenck, T. Szemberg, and Z. Teithler,
\emph{Recent developments and open problems in linear series,}
In: Contributions to Algebraic Geometry, p. 93-140, EMS Ser. Congr. Rep., Eur. Math. Soc., Z\"{u}rich, 2012.

\bibitem{Bauer et al. 2013}
T. Bauer, B. Harbourne, A. L. Knutsen, A. K$\mathrm{\ddot{u}}$ronya, S. M$\mathrm{\ddot{u}}$ller-Stach, X. Roulleau, and T. Szemberg,
\emph{Negative curves on algebraic surfaces,}
Duke Math. J. \textbf{162}(10)(2013), 1877-1894.

\bibitem{Ciliberto et al. 2017}
C. Ciliberto, A. L. Knutsen, J. Lesieutre, V. Lozovanu, R. Miranda, Y. Mustopa, and D. Testa,
\emph{A few questions about curves on surfaces,} Rend. Circ. Mat. Palermo. II. Ser. \textbf{66} (2)(2017),195-204.

\bibitem{CHMR13}
C. Ciliberto, B. Harbourne, R. Miranda, and J. Ro\'e,
\emph{Variations on Nagata's Conjecture},
Clay Math. Proc. \textbf{18} (2013), 185-203.

\bibitem{Esole17}
M. Esole,
\emph{Introduction to elliptic fibrations}, Quantization, Geometry and Noncommutative Structures in Mathematics and Physics, 
Springer International Publishing, Cham. pp. 247-276.

\bibitem{HL26}
Z. Hua and S. Li, 
\emph{Smooth projective surfaces with bounded cohomology property},  Internat. J. Math. (2026), \arxiv{2306.07830v5}.
  
 \bibitem{Humphreys72}
 J. Humphreys, 
 \emph{Introduction to Lie algebras and representation theory},
GTM \textbf{9}, Springer-Verlag, New York, 1972.

  \bibitem{Kodaira63}
K. Kodaira,
\emph{On compact analytic surfaces II-III},
Ann. Math. \textbf{77} (1963), 563-626.

\bibitem{Lazarsfeld04}
R. Lazarsfeld, \emph{Positivity in algebraic geometry, I}, Ergebnisse der
  Mathematik und ihrer Grenzgebiete. \textbf{48}, Springer-Verlag, Berlin, 2004.
  
 \bibitem{Li19}
S. Li,
\emph{A note on a smooth projective surface with Picard number 2,}
Math. Nachr. \textbf{292} (2019), no.~12, 2637-2642.

\bibitem{Li21}
S. Li,
\emph{Bounding cohomology on a smooth projective surface with Picard number 2,}
Commun. Algebra \textbf{49} (2021), no. 7, 3140-3144.

\bibitem{Li23}
S. Li,
\emph{Bounded cohomology property on a smooth projective surface with Picard number two}, Commun. Algebra. \textbf{51} (2023), no. 12, 5235-5241.

\bibitem{Li26}
S. Li,
\emph{Smooth projective surfaces with bounded cohomology property II}, 
\arxiv{2608.12005}.

\bibitem{LLL26}
A. Laface, S. Li and J. Liu,
\emph{Mori dream Jacobian elliptic surfaces of Kodaira dimension one}, \arxiv{2608.09002v2}.

\bibitem{Nikulin00}
V. Nikulin,
\emph{A remark on algebraic surfaces with polyhedral Mori cone},
Nagoya Math. J. \textbf{157} (2000), 73-92.

\bibitem{SS19}
M. Sch\"utt and T. Shioda,
\emph{Mordell-Weil lattices},
 Ergeb. Math. Grenzgeb.  \textbf{70}. Springer, Singapore, 2019. 
 
\bibitem{Nagata59}
M. Nagata,
\emph{On the 14th problem of Hilbert},
Amer. J. Math. \textbf{81} (1959),766-772.

\end{thebibliography}
\end{document}